\documentclass{article}
\usepackage[latin1]{inputenc}
\usepackage{amsmath, amsthm, amsfonts, mathrsfs, MnSymbol}

\newtheorem{Thm}{Theorem}
\newtheorem{Lemma}[Thm]{Lemma}

\newtheorem{Cor}[Thm]{Corollary}
\newtheorem*{Ex}{Example}

\theoremstyle{definition}
\newtheorem{Def}[Thm]{Definition}
\theoremstyle{remark}
\newtheorem*{Remark}{Remark}

\title{Group gradings on the Jordan algebra of upper triangular matrices}
\author{Plamen Emilov Koshlukov\thanks{Partially supported by FAPESP grant No. 2014/09310-5, and by CNPq grant No. 304632/2015-5}\and Felipe Yukihide Yasumura\thanks{Supported by PhD grant 2013/22802-1 from FAPESP}\\
Department of Mathematics, State University of Campinas\\
651 Sergio Buarque de Holanda\\
13083-859 Campinas, SP, Brazil\\
e-mail addresses: \texttt{plamen,  	ra091138@ime.unicamp.br}}
\date{}

\begin{document}
\maketitle
\begin{abstract}
Let $G$ be an arbitrary group and let $K$ be a field of characteristic different from 2. We classify the $G$-gradings on the Jordan algebra $\text{UJ}_n$ of upper triangular matrices of order $n$ over $K$. It turns out that there are, up to a graded isomorphism, two families of gradings: the elementary gradings (analogous to the ones in the associative case), and the so called mirror type (MT) gradings. Moreover we prove that the $G$-gradings on $\text{UJ}_n$ are uniquely determined, up to a graded isomorphism, by the graded identities they satisfy. 
\end{abstract}



\section*{Introduction}
Graded algebras were first introduced and studied in Commutative algebra in order to generalize properties of polynomial rings. The study of the gradings on associative algebras was initiated by Wall \cite{wall}. He described the finite dimensional graded simple algebras with grading group $\mathbb{Z}_2$, the cyclic group of order 2. Much later, around 1985, Kemer developed the structure theory of the T-ideals (ideals of identities) in the free associative algebra, see for instance \cite{kemerbook}. One of the principal ingredients of that theory is the study of $\mathbb{Z}_2$-graded algebras and their graded identities. The theory developed by Kemer has since immensely influenced the research in PI theory, and motivated further study of gradings and on graded polynomial identities in associative algebras. Motivated in part by Kemer's results, gradings on algebras became an object of extensive study. We cite here the recent monograph \cite{ElKo2013} and the bibliography therein for further and more detailed information concerning gradings on algebras. We recall some of the cornerstone results in the area that will be used in the exposition below. In order to do it in a precise way we have to introduce now part of the notions we will need. 

Let $G$ be a group, $K$ a field, and $A$ an algebra (that is a vector space over $K$ equipped with a bilinear multiplication). We do not require the multiplication in $A$ to be either commutative or associative. The algebra $A$ is $G$-graded if $A=\oplus_{g\in G}A_g$ where the subspaces $A_g\subseteq A$ satisfy $A_gA_h\subseteq A_{gh}$,                                                                     $g$, $h\in G$. A vector subspace $V$ of $A$ is homogeneous, or graded, if $V=\oplus_{g\in G} V\cap A_g$. In the same way one defines graded subalgebras and ideals of $A$. 

The gradings on the (associative) matrix algebras of order $n$ were completely described by Bahturin and Zaicev, see for example \cite{bahtzai03}. Later on this description was extended, in the case of abelian groups, to simple associative algebras with a minimal one sided ideal (over an algebraically closed field), see \cite[pp. 27, 28]{ElKo2013}. Similar results were obtained for simple Lie algebras, see for example \cite{ElKo2013};  simple Jordan algebras, see \cite{bashza} and its bibliography.

We note that relatively little is known about the classification of the gradings on important algebras that are not simple. The Grassmann (or exterior) algebra appears naturally in various branches of Mathematics and Physics but the gradings on this algebra are known under restrictions on the gradings \cite{omdvvs}. The gradings on the Jordan algebra of the $2\times 2$ upper triangular matrices were described in \cite{pkfm}. Note that the latter algebra is a Jordan algebra of a symmetric bilinear form though the form is degenerate.

The elementary gradings on the associative algebra $UT_n(K)$ of the upper triangular matrices of order $n$ over a field $K$ were described in \cite{VinKoVa2004}. Also, in \cite{VaZa2007}, the authors proved that every grading on $UT_n$ is, up to a graded isomorphism, elementary. Hence one has a complete classification of the gradings on the associative algebra of upper triangular matrices.

In this paper we study the gradings on the Jordan algebra $\text{UJ}_n$ of the upper triangular matrices of order $n$ over an infinite field. We describe completely all these gradings. Our principal result states that every group grading on $\text{UJ}_n$ is either elementary or is of the so-called mirror type (MT for short).

For the particular case $n=2$, we recall the description of the gradings on $\text{UJ}_2$ given in \cite{pkfm}. 
\begin{Thm}[\cite{pkfm}]
	Denote $1=e_{11}+e_{22}$, $a=e_{11}-e_{22}$, $b=e_{12}\in\text{UJ}_2$. Let $\text{UJ}_2=J_0+J_1$ be $\mathbb{Z}_2$-graded. Then the grading is isomorphic to one of the following gradings: 
	\begin{enumerate}
		\setcounter{enumi}{-1}
		\item The trivial grading: $J_0=\text{UJ}_2$;
		\item The associative grading: $J_0=1K\oplus bK$, $J_1=aK$;
		\item The scalar grading: $J_0=1K$, $J_1=aK\oplus bK$;
		\item The classical grading: $J_0=1K\oplus aK$, $J_1=bK$.
	\end{enumerate}
\end{Thm}
The above theorem will be a particular case of our results. We recall that in \cite{pkfm} the authors also described the graded polynomial identities satisfied by each of the possible gradings, including the trivial one. Here we are not going to do that. It seems to us that the description of the graded identities satisfied by $\text{UJ}_n$ is a very complicated problem, and even in the simplest cases it turned out to be far from our reach.

\section{Notations}
Here, we will denote $\text{UJ}_n=UT_n^{(+)}$ the vector space of the upper triangular matrices of size $n$ equipped with the Jordan product $a\circ b=ab+ba$.
We denote the associator by $(a,b,c)=(a\circ b)\circ c-a\circ(b\circ c)$.
We shall omit the parentheses in left normed products, that is $a\circ b\circ c$ stands for $(a\circ b)\circ c$, and similarly for products of more than 3 elements.

We denote by $e_{ij}$ the matrix units having entry 1 at position $(i,j)$, and 0 elsewhere. Given a matrix $x\in\text{UJ}_n$ we denote by $(x)_{(i,j)}$ the $(i,j)$-th entry of $x$. Also for $i$, $m\in\mathbb{N}$ we define 
\begin{eqnarray*}
	e_{i:m}=e_{i,i+m},&& e_{-i:m}=e_{n-i-m+1,n-i+1},\\
	(x)_{(i:m)}=(x)_{(i,i+m)},&& (x)_{(-i:m)}=(x)_{(n-i-m+1,n-i+1)}.
\end{eqnarray*}

Let $G$ be a group with neutral element $1$. Unless otherwise stated we use the multiplicative notation, even if the group is abelian. Of course we write additively the groups $(\mathbb{Z},+)$, and $(\mathbb{Z}_n,+)$.

First we classify the elementary and MT gradings (to be defined below), and then prove that every grading on $\text{UJ}_n$ is, up to a graded isomorphism, either elementary or MT.

\section{Elementary Gradings}

Let $G$ be any group. We call a $G$-grading on $\text{UJ}_n$ elementary if all matrix units $e_{ij}$ are homogeneous in the grading.
\begin{Lemma}\label{basic_element}
Let $\text{UJ}_n$ be equipped with an elementary $G$-grading. Then
		\begin{enumerate}
			\renewcommand{\labelenumi}{(\roman{enumi})}
			\item $\deg e_{ii}=1$, $i=1$, \dots, $n$.
			\item The sequence $\eta=(\deg e_{12}, \deg e_{23},\ldots, \deg e_{n-1,n})$  defines completely the grading.
			\item The support of the grading is commutative.
		\end{enumerate}
	\end{Lemma}
	\begin{proof}
The statements of the lemma and their proofs are standard facts, we give these proofs for the sake of completeness.

(i) Since $e_{ii}\circ e_{ii}=2e_{ii}$ we have $(\deg e_{ii})^2=\deg e_{ii}$ hence $\deg e_{ii}=1$.

(ii) It follows from $	e_{ij}=e_{i,i+1}\circ e_{i+1,i+2}\circ\cdots\circ e_{j-1,j}$.

(iii) Let $t_1=\deg e_{12}$, $t_2= \deg e_{23}$, \dots, $t_{n-1}= \deg e_{n-1,n}$. By (ii), it suffices to prove that $t_it_j=t_jt_i$ for all $i$, $j\in\{1,2,\ldots,n-1\}$. But if $i<j$ then
\[
				e_{i,i+1}\circ (e_{j,j+1}\circ(e_{i+1,i+2}\circ\cdots\circ e_{j-1,j}))=e_{j,j+1}\circ(e_{i,i+1}\circ e_{i+1,i+2}\circ\cdots\circ e_{j-1,j}).
\]
Thus $t_it_jt_{i+1}\cdots t_{j-1}=t_jt_it_{i+1}\cdots t_{j-1}$ and $t_it_j=t_jt_i$.
	\end{proof}

	From here on in this section, we assume that $G$ is abelian.
	
	\noindent\textbf{Notation.} We denote by $(\text{UJ}_n,\eta)$ the elementary grading defined by $\eta\in G^{n-1}$. This grading is defined by putting $\deg e_{i,i+1}=g_i$, for each $i$, where $\eta=(g_1,g_2,\ldots,g_{n-1})$. We denote by $\text{rev}\,\eta=(g_{n-1},g_{n-2},\ldots,g_1)$.

	\begin{Lemma}
		Let $\eta\in G^{n-1}$. The map $\varphi\colon (\text{UJ}_n,\eta)\to (\text{UJ}_n,\text{rev}\,\eta)$ given by $e_{ij}\mapsto e_{n-j+1,n-i+1}$ is an isomorphism of $G$-graded algebras.
	\end{Lemma}
	\begin{proof}
		The proof is a direct and easy verification.
	\end{proof}
	
We recall notations and results of \cite{BL1992, HY}. Denote by $S_m$ the set of permutations of $m$ elements and let 
\[
		\mathscr{T}_m=\{\sigma\in S_m\mid\sigma(1)>\cdots>\sigma(t)=1,\sigma(t+1)<\cdots<\sigma(m)\}.
\]
	Using same argument as \cite[Lemma 3 (ii)]{HY}, one can prove
	\begin{Lemma}\label{lemmaum}
		Let $r_1$, \dots, $r_m$ be strictly upper triangular matrix units such that the associative product $r_1\cdots r_m\ne0$, and let $\sigma\in S_m$. Then $r_{\sigma^{-1}(1)}\circ\cdots\circ r_{\sigma^{-1}(m)}\ne0$ if and only if $\sigma\in\mathscr{T}_m$.
	\end{Lemma}

In analogy with \cite{VinKoVa2004} we define

	\begin{Def}
Let $G$ be a group and let $(\text{UJ}_n,\eta)$ be an elementary $G$-grading. Let $\mu=(a_1,\ldots,a_m)\in G^m$ be any sequence.
		\begin{enumerate}
			\item (See \cite{VinKoVa2004}) The sequence $\mu$ is associative $\eta$-good if there exist strictly upper triangular matrix units $r_1$, \dots, $r_m\in UT_n$ such that $r_1\cdots r_m\ne0$ and $\deg r_i = a_i$ for every $i=1$, \dots, $m$. Otherwise $\mu$ is associative $\eta$-bad sequence.
			
			\item The sequence $\mu$ is Jordan $\eta$-good if there exist strictly upper triangular matrix units $r_1$, \dots, $r_m$ such that $r_1\circ\cdots\circ r_m\ne0$ and $\deg r_i=a_i$, for every $i=1$, \dots, $m$. Otherwise $\mu$ is Jordan $\eta$-bad sequence.
		\end{enumerate}
	\end{Def}
	
	\begin{Def}
If $\mu=(a_1,a_2,\cdots,a_m)\in G^m$ we define
\[
			f_\mu=f_1^{(a_1)}\circ f_2^{(a_2)}\circ\cdots\circ f_m^{(a_m)}
\]
where
\[
			f_h^{(a)}=\left\{\begin{array}{ll}
				(x_{3h-2}^{(1)},x_{3h-1}^{(1)},x_{3h}^{(1)}),&\text{if $a=1$},\\
				x_h^{(a)},&\text{if $a\ne1$}
			\end{array}\right..
\]
	\end{Def}

The following lemma is proved exactly in the same way as Proposition 2.2 of \cite{VinKoVa2004}.
	\begin{Lemma}\label{lemmadois}
		A sequence $\mu$ is Jordan $\eta$-bad if and only if $f_\mu$ is a $G$-graded identity for $(\text{UJ}_n,\eta)$.
	\end{Lemma}

If $S$ is any set and $s=(s_1,\ldots,s_m)\in S^m$ is any sequence of symbols, we define the left action of $S_m$ on $S^m$ by
\[
		\sigma s=(s_{\sigma^{-1}(1)},\ldots,s_{\sigma^{-1}(m)}),\sigma\in S_m.
\]
	The unique non-zero associative product of $n-1$ strictly upper triangular matrix units of $UT_n$ is $e_{12}e_{23}\cdots e_{n-1,n}$ (see \cite{VinKoVa2004}), so combining this fact, Lemma~\ref{lemmaum}, and Lemma~\ref{lemmadois}, we obtain
	\begin{Lemma}\label{lemmatres}
		A sequence $\mu\in G^{n-1}$ is Jordan $\eta$-good for $(\text{UJ}_n,\eta)$ if and only if $\mu=\sigma\eta$ for some $\sigma\in\mathscr{T}_{n-1}$.
	\end{Lemma}
The following lemma was proved in \cite{HY}.
	\begin{Lemma}[\cite{HY}]\label{lemmaquatro}
		Let $s$, $s'\in S^m$ be any sequences where $S$ is a set. Then $s=s'$ or $s=\text{rev}\,s'$ if and only if for every $\sigma$, $\tau'\in\mathscr{T}_{n-1}$ we can find $\sigma'$, $\tau\in\mathscr{T}_{n-1}$ such that $\sigma s=\sigma' s'$ and $\tau s=\tau's'$.
	\end{Lemma}

	Combining Lemmas \ref{lemmatres} and \ref{lemmaquatro}, we obtain
	\begin{Cor}
		Let $\eta$, $\eta'\in G^{n-1}$ with $\eta\ne\eta'$ and $\eta\ne\text{rev}\,\eta'$. Then $(\text{UJ}_n,\eta)\not\simeq(\text{UJ}_n,\eta')$.
	\end{Cor}
	\begin{proof}
		By Lemma \ref{lemmaquatro}, there exists $\sigma\in\mathscr{T}_m$ such that $\sigma\eta\ne\sigma'\eta'$ for each $\sigma'\in\mathscr{T}_m$, interchanging $\eta$ and $\eta'$ if necessary. By Lemma \ref{lemmatres}, $\sigma\eta$ is Jordan $\eta$-good sequence but Jordan $\eta'$-bad sequence, hence $f_{\sigma\eta}$ is not a graded identity for $(\text{UJ}_n,\eta)$, but it is a graded identity for $(\text{UJ}_n,\eta')$. In particular, $(\text{UJ}_n,\eta)\not\simeq(\text{UJ}_n,\eta')$.
	\end{proof}

In this way we have a classification of the elementary gradings on $\text{UJ}_n$:
	\begin{Thm}\label{class_element}
		The support of an elementary $G$-grading on $\text{UJ}_n$ is commutative.

		Let $G$ be an abelian group and define the equivalence relation on $G^{n-1}$ as follows. Let $\mu_1$ and $\mu_2=(a_1,a_2,\ldots,a_{n-1})\in G^{n-1}$, then $\mu_1\sim\mu_2$ whenever $\mu_1=\mu_2$ or $\mu_1=(a_{n-1},\cdots,a_2,a_1)$.

		Then there is 1--1 correspondence between $G^{n-1}/\sim$ and the class of non-isomorphic elementary $G$-gradings on $\text{UJ}_n$.
	\end{Thm}

\section{MT Gradings}
	\noindent\textbf{Notation.} If $i$, $m\in\mathbb{N}$ we denote $Y_{i:m}^+ =e_{i:m}+e_{-i:m}$, and $Y_{i:m}^-=e_{i:m}-e_{-i:m}$.

	\begin{Remark}
In the above notation, if $n-m$ is odd then $Y_{i:m}^+=2e_{i:m}=2e_{-i:m}$, and $Y_{i:m}^-=0$ for $i=(n-m+1)/2$.
	\end{Remark}

	\begin{Def}
A $G$-grading on $\text{UJ}_n$ is of mirror pattern type (or just MT) if all $Y_{i:m}^+$, $Y_{i:m}^-$ are homogeneous and $\deg Y_{i:m}^+\ne \deg Y_{i:m}^-$.
	\end{Def}

	\begin{Lemma}\label{multi_table}
One has $	Y_{i:1}^+\circ Y_{i+1:1}^+ \circ\cdots\circ Y_{i+m-1:1}^+ = \lambda Y_{i:m}^+$ for some $\lambda=2^p$,  $p\in\mathbb{Z}$.
	\end{Lemma}
	\begin{proof}
Induction on $m$. When $m=1$ the statement is trivial, so assume $m>1$. If
$Y_{i:1}^+\circ Y_{i+1:1}^+ \circ\cdots\circ Y_{i+m-1:1}^+=\lambda Y_{i:m}^+=\lambda(e_{i:m}+e_{-i:m})$ then $(\lambda Y_{i:m}^+)\circ Y_{i+m:1}^+=\lambda'(e_{i:m+1}+e_{-i:m+1})$.
	\end{proof}

	\begin{Lemma}
Let a $G$-grading on $\text{UJ}_n$ be MT, then
		\begin{enumerate}
			\renewcommand{\labelenumi}{(\roman{enumi})}
			\item $ \deg Y_{i:0}^+=1$ for every $i$, and $ \deg Y_{1:0}^-= \deg Y_{2:0}^-=\cdots= \deg Y_{\lfloor\frac{n}2\rfloor:0}^-=t$ is an element of order 2.
			\item Let $q=\left\lceil\frac{n-1}2\right\rceil$, then the sequence $\eta=( \deg Y_{1:1}^+, \deg Y_{2:1}^+, \ldots, \deg Y_{q:1}^+)$ and the element $t= \deg Y_{1:0}^-$ completely define the grading.
			\item The support of the grading is commutative.
		\end{enumerate}
		Moreover, if the elements $Y_{i:m}^\pm$, for each  $i$ and for $m=0$ and $m=1$, are homogeneous, then the grading is necessarily MT.
	\end{Lemma}
	\begin{proof}
		(i) The equalities $Y_{1:0}^-\circ Y_{1:1}^+=Y_{1:1}^-$, $(Y_{i:0}^\pm)^2=2Y_{i:0}^+$ and $Y_{i:0}^-\circ Y_{i:1}^\pm=Y_{i+1:0}^-\circ Y_{i:1}^\pm$ yield the proof.
		
		(ii) It follows from Lemma \ref{multi_table}.
		
		(iii) According to (ii), the elements $\deg Y_{1:0}^-$ and $\deg Y_{i:1}^+$, for all $i$, generate the support of the group. Using Lemma \ref{multi_table} and the same idea as of Lemma \ref{basic_element}.(iii), we  prove the statement.
	\end{proof}

We assume from now on in this section $G$ abelian. We denote by $(\text{UJ}_n,t,\eta)$ the MT-grading defined by $t\in G$ and the sequence $\eta\in G^q$.

It is well known that, if we have an associative algebra with involution $(A,\ast)$, then the decomposition of $A$ into symmetric and skew-symmetric elements with respect to $\ast$ gives rise to a $\mathbb{Z}_2$-graded algebra. If, moreover, $A$ is endowed with an $H$-grading and $\ast$ is a graded involution (that is, $\deg a^\ast=\deg a$, for all homogeneous $a\in A$), then the decomposition cited yields an $H\times\mathbb{Z}_2$-graded Jordan algebra.

The upper triangular matrices possess a natural involution, given by $\psi:e_{i:m}\in UT_n\mapsto e_{-i:m}\in UT_n$. For an elementary grading $\eta$ on $UT_n$, $\psi$ will be a graded involution if and only if $\eta=\text{rev}\,\eta$. It is easy to see that the obtained grading by the involution is an MT-grading.

The following example, in its present form, was suggested by the Referee. 
\begin{Ex}
	Let $G=\mathbb{Z}_4$ and take the MT-grading on $\text{UJ}_4$ given by $\deg Y_{i:1}^+=1\in\mathbb{Z}_4$, and $\deg Y_{i:0}^-=2\in\mathbb{Z}_4$, for every $i$. Since $\mathbb{Z}_4$ is an indecomposable group, it cannot be written in the form $\mathbb{Z}_2\times H$. Therefore there exist MT-gradings that cannot be given by the involution.
\end{Ex}

Below we classify all MT gradings. We recall that, according to \cite{BBC}, every automorphism of $\text{UJ}_n$ is given either by an automorphism of $UT_n$, or by an automorphism of $UT_n$ followed by the involution $e_{i:m}\mapsto e_{-i:m}$. In particular, the ideal $J$ of all strictly upper triangular matrices is invariant under all automorphisms of $\text{UJ}_n$.

	\begin{Lemma}
		Let $\eta$, $\eta'\in G^q$ where $q=\lceil\frac{n-1}2\rceil$ and $t_1$, $t_2\in G$ are elements of order 2. If $t_1\ne t_2$ then $(\text{UJ}_n,t_1,\eta)\not\simeq(\text{UJ}_n,t_2,\eta')$.
	\end{Lemma}
	\label{afterexample}
	\begin{proof}
		If $\psi\colon (\text{UJ}_n,t_1,\eta)\to(\text{UJ}_n,t_2,\eta')$ is a graded isomorphism then $\bar\psi\colon\text{UJ}_n/J\to\text{UJ}_n/J$ will be a graded isomorphism which is impossible when $t_1\ne t_2$.
	\end{proof}
	\begin{Lemma}\label{notation_bla}
		Let $t\in G$ be an element of order 2 and let $\eta=(g_1,\ldots,g_q)$, $\eta'=(g_1',\ldots,g_q')\in G^q$ where $q=\lceil\frac{n-1}2\rceil$. Assume that one of the following holds:
		\begin{itemize}
			\item there is an $i$, $1\le i\le\lfloor\frac{n-1}2\rfloor$ such that $g_i\not\equiv g_i'\pmod{\langle t\rangle}$, or
			\item $n$ is even and $g_q\ne g_q'$.
		\end{itemize}
		Then $(\text{UJ}_n,t,\eta)\not\simeq(\text{UJ}_n,t,\eta')$.
	\end{Lemma}
	\begin{proof}
		Let $\varphi\colon G\to G_0=G/\langle t\rangle$ be the canonical projection. The induced $G_0$-grading on $\text{UJ}_n$ by $\varphi$ and by  $(\text{UJ}_n,t,\eta)$ coincides with the elementary $G_0$-grading $(\text{UJ}_n,\eta_0)$ where
\[
			\eta_0=\left\{\begin{array}{l}
				(\varphi(g_1),\varphi(g_2),\ldots,\varphi(g_q),\varphi(g_q),\varphi(g_{q-1}),\ldots,\varphi(g_1)),\text{ if $n$ is odd},\\
				(\varphi(g_1),\varphi(g_2),\ldots,\varphi(g_{q-1}),\varphi(g_q),\varphi(g_{q-1}),\ldots,\varphi(g_1)),\text{ if $n$ is even}.
			\end{array}\right.
\]
		A $G$-graded isomorphism $\psi\colon (\text{UJ}_n,t,\eta)\to(\text{UJ}_n,t,\eta')$  induces a $G_0$-graded isomorphism $(\text{UJ}_n,\eta_0)\to(\text{UJ}_n,\eta_0')$ if and only if $\eta_0=\eta_0'$ (since $\eta_0'=\text{rev}\,\eta_0'$), by Theorem \ref{class_element}. This proves the first condition.

		Now, assume $n$ even and $g_q\ne g_q'$. Let $T=\text{Span}\{Y_{i:m}^\pm\mid(i,m)\not\in\{(q,1),(q,0)\}\}$ (note that $T$ is invariant under all automorphisms of $\text{UJ}_n$). Then $T$ is a graded ideal, and $\text{UJ}_n/T\simeq\text{UJ}_2$. But $(\text{UJ}_n/T,t,\eta)\not\simeq(\text{UJ}_n/T,t,\eta')$ if $g_q\ne g_q'$.
	\end{proof}

	\begin{Lemma}
		Let $t\in G$ be of order 2, $\eta=(g_1,\ldots,g_q)$, $\eta'=(g_1,\ldots,g_q')\in G^q$ where $q=\lceil\frac{n-1}2\rceil$. Assume that
		\begin{enumerate}\renewcommand{\labelenumi}{\roman{enumi})}
			\item $g_i\equiv g_i'\pmod{\langle t\rangle}$, for $i=1$, 2, \dots, $p$ where $p=\lfloor\frac{n-1}2\rfloor$,
			\item if $n$ is even then $g_q=g_q'$.
		\end{enumerate}
		Then $(\text{UJ}_n,t,\eta)\simeq(\text{UJ}_n,t,\eta')$.
	\end{Lemma}
	\begin{proof}
		For every $i=1$, 2, \dots, $p$, let $\epsilon_i=1$ if $g_i=g_i'$ and $\epsilon_i=-1$ if $g_i\ne g_i'$. Let $\epsilon=\epsilon_1\epsilon_2\cdots\epsilon_p$ and
$A=\text{diag}(\epsilon,\epsilon,\ldots,\epsilon, \epsilon_{p-1} \epsilon_{p-2}\cdots \epsilon_1, \epsilon_{p-2} \cdots\epsilon_1, \ldots,\epsilon_2 \epsilon_1,\epsilon_1,1)$. The map $(\text{UJ}_n,t,\eta)\to (\text{UJ}_n,t,\eta')$ given by $x\mapsto AxA^{-1}$ is a graded isomorphism.
	\end{proof}

The following theorem classifies the MT gradings on $\text{UJ}_n$.
	\begin{Thm}
		Every MT grading has commutative support. If $G$ is abelian, then there is  1--1 correspondence between the  non-isomorphic MT gradings on $\text{UJ}_n$ and the set $M$ where
		\begin{enumerate}
			\item if $n$ is odd, $M=\{(t,\eta)\mid t\in G,o(t)=2,\eta\in(G/\langle t\rangle)^{\frac{n-1}2}\}$,
			\item if $n$ is even, $M=\{(t,\eta)\mid t\in G,o(t)=2,\eta\in(G/\langle t\rangle)^{\frac{n-2}2}\times G\}$.
		\end{enumerate}
	\end{Thm}

\section{General gradings on $\text{UJ}_n$}
	Let $G$ be any group and fix a $G$-grading on $\text{UJ}_n$. The ideal $J$ of all strictly upper triangular matrices is homogeneous since $J=(\text{UJ}_n,\text{UJ}_n,\text{UJ}_n)$. Also the element $e_{1n}$ is always homogeneous since $\text{Span}\{e_{1n}\}=J^{n-1}$. 

	As a consequence $B=\text{Ann}_{\text{UJ}_n}\{e_{1n}\}=\{x\in\text{UJ}_n\mid (x)_{(1,1)}+(x)_{(n,n)}=0\}$ is homogeneous, and $B^2=B\circ B=\{x\in\text{UJ}_n\mid(x)_{(1,1)}=(x)_{(n,n)}\}$ is as well. It follows $C=B\cap B^2=\{x\in\text{UJ}_n\mid(x)_{(1,1)}=(x)_{(n,n)}=0\}$ is homogeneous. Let $U_1=\text{Ann}_{\text{UJ}_n}(C/J)$ and let $T_1=U_1^{\circ n}=\{x\in\text{UJ}_n\mid (x)_{(i,j)}=0,\text{ for $i\ne1$ or $(i,j)\ne (i,n)$}\}$, the $n$-th power of $U_1$. It is easy to see that $T_1$ is an ideal (moreover, a graded ideal). A similar trick in the associative case can be found in the proof of Lemma 2 of  \cite{VaZa2003}.

	\begin{Lemma}\label{prevlemma}
		There exists a homogeneous element $e_2\in T_1$ such that $(\deg e_2)^2=1$ and $e_2\equiv e_{11}-e_{nn}\pmod{T_1\cap J}$.
	\end{Lemma}
	\begin{proof}
Note first that $A=T_1/T_1\cap J$ is an associative graded algebra whose unit is $\bar e_1=\bar e_{11}+\bar e_{nn}$. Hence $\bar e_1$ is graded and $\deg\bar e_1=1$. Moreover, we can choose the homogeneous element $x\in T_1$ and we can assume that $\bar x$ and $\bar e_1$ are linearly independent in $A$. If $\deg\bar x=1$ then we are done. Otherwise $\deg\bar x\ne\deg (\bar x\circ\bar x)$ which implies  $\bar x\circ\bar x$ is a multiple of $\bar e_1$, and this proves the lemma.
	\end{proof}

	\begin{Lemma}\label{lemmabefore}
		Up to a graded isomorphism, $e_1=e_{11}+e_{nn}$ and $e_2=e_{11}-e_{nn}$ are homogeneous and $\deg e_1=(\deg e_2)^2=1$.
	\end{Lemma}
	\begin{proof}
		Let $e_2$ be as in the previous lemma, and let $e_1=\frac12e_2\circ e_2$. Note that
		\begin{enumerate}
			\renewcommand{\labelenumi}{(\alph{enumi})}
			\item $e_1\equiv e_{11}+e_{nn}\pmod{T_1\cap J}$.
			\item $(e_1)_{(1,i)}=(e_2)_{(1,i)}$ and $(e_1)_{(i,n)}=-(e_2)_{(i,n)}$, for $i=2$, 3, \dots, $n-1$.
			\item $(e_1)_{(1,n)}=\sum_{i=2}^{n-1}(e_2)_{(1,i)}(e_2)_{(i,n)}$.
		\end{enumerate}
		As a consequence of the above properties, the associative product $x=e_1(e_1-1)=0$. Indeed,
		\begin{enumerate}
			\renewcommand{\labelenumi}{(\alph{enumi})}
			\item $(x)_{(1,i)}=(e_1)_{(1,1)}(e_1-1)_{(1,i)}+(e_1)_{(1,i)}(e_1-1)_{(i,i)}=0$, for every $i=1$, 2, \dots, $n-1$.
			\item $(x)_{(i,n)}=0$, for $i=2,3,\ldots,n-1$.
			\item Using the above relations one obtains 
				\begin{align*}(x)_{(1,n)}&=\sum_{i=1}^n(e_1)_{(1,i)}(e_1-1)_{(i,n)}\\
							&=\underbrace{(e_1)_{(1,1)}(e_1-1)_{(1,n)}}_{(e_1)_{(1,n)}}+\sum_{i=2}^{n-1}(e_1)_{(1,i)}\underbrace{(e_1-1)_{(i,n)}}_{(e_1)_{(i,n)}}=0.
				\end{align*}
			\item All remaining entries are evidently zero.
		\end{enumerate}
		These equalities show that the minimal polynomial of $e_1$ is $z(z-1)$, hence $e_1$ is diagonalizable. If $\psi\colon\text{UJ}_n\to\text{UJ}_n$ is the conjugation such that $\psi(e_1)=e_{11}+e_{nn}$, then $\psi$ induces a new $G$-grading on $\text{UJ}_n$, isomorphic to the original one, such that $e_1=e_{11}+e_{nn}$ is homogeneous of degree $1$.

Consider again the  element $e_2$ from Lemma \ref{prevlemma}. Let $r_2=e_2\circ e_1-e_2$. Then $r_2= e_{11}-e_{nn}+\alpha e_{1n}$ for some $\alpha\in F$, and moreover, $r_2$ is diagonalizable. Since $e_1$ and $r_2$ commute, they are simultaneously diagonalizable, and we can find an inner automorphism $\psi'$ such that $\psi'(e_1)=e_{11}+e_{nn}$ and $\psi'(r_2)=e_{11}-e_{nn}$. This concludes the lemma.
	\end{proof}

If $(A,\circ)$ is a Jordan algebra, not necessarily with unit and $x\in A$, we denote $(1-x)\circ A=\{y-x\circ y\mid y\in A\}$. Suppose the element $e_1=e_{11}+e_{nn}\in\text{UJ}_n$ is homogeneous. Then the following set is also homogeneous:
\[
\Delta=(1-e_1)\circ\left((1-\frac12e_1)\circ\text{UJ}_n\right)=\text{UJ}_{n-2}.
\]

Thus we write $\text{UJ}_n=T_1\oplus\Delta$. Note that every inner automorphism (conjugation) of $\Delta$ by a matrix $M$ can be extended to an inner automorphism of $\text{UJ}_n$ by the matrix
\[
		M'=\begin{pmatrix}
			1&0&0\\
			0&M&0\\
			0&0&1
		\end{pmatrix}.
\]
Therefore we can repeat the argument above for $\Delta$. Thus we suppose that, up to a graded isomorphism, the elements $u_1=e_{22}+e_{n-1,n-1}$ and $u_2=e_{22}-e_{n-1,n-1}$ are homogeneous and $\deg u_1=(\deg u_2)^2=1$. Since $M'e_i(M')^{-1}=e_i$ for $i=1$, 2, we can also assume  the existence of the elements $e_1$ and $e_2$ as in Lemma~\ref{lemmabefore}.

Take a homogeneous element $z_1\in J^{n-2}=\text{Span}\{e_{1,n-1},e_{2n},e_{1n}\}$ such that $(z)_{(1,n-1)}\ne 0$. We can change $z$ to $z\circ u_1$, if necessary, in order to obtain $(z)_{(1,n)}=0$. In this case $u_2\circ z=-e_2\circ z$, hence we have
	\begin{Lemma}\label{lemmaanterior}
		In the notation introduced above, $\deg e_2=\deg u_2$.
	\end{Lemma}

If $\deg e_2=1$ then $e_{11}$ and $e_{nn}$ are homogeneous of degree 1.
	\begin{Lemma}
		If $\deg e_2=1$ then, up to a graded isomorphism, the grading is elementary.
	\end{Lemma}
	\begin{proof}
		If $n=2$ then the elements $e_{11}$, $e_{22}$, $e_{12}$ are (up to a graded isomorphism) homogeneous hence the grading is elementary. If $n=3$ we consider the decomposition $\text{UJ}_3=T_1\oplus\Delta$. Since $\dim\Delta=1$ it is easy to prove that the grading is again elementary.

Thus we assume $n>3$. We decompose $\text{UJ}_n=T_1\oplus\Delta$. In the notation introduced above, by Lemma \ref{lemmaanterior} we have $\deg u_2=\deg e_2=1$. We use an induction to conclude that the grading on $\Delta$ is elementary. In particular the elements $e_{22}$ and $e_{n-1,n-1}$ are homogeneous. If $z\in J$ is homogeneous with $(z)_{(1,2)}=1$ then $e_{12}=(z\circ e_{11})\circ e_{22}$ is homogeneous. In the same way we obtain that $e_{n-1,n}$ is homogeneous. This implies that the elements $e_{12}$, $e_{23}$, \dots, $e_{n-1,n}$ are homogeneous. Therefore the grading is elementary.
	\end{proof}

	\begin{Lemma}
		If $\deg e_2\ne1$ then, up to a graded isomorphism, the grading is MT.
	\end{Lemma}
	\begin{proof}
First we assume $n>3$. Decompose $\text{UJ}_n=T_1\oplus\Delta$, by Lemma \ref{lemmaanterior}, we  assume by induction that up to a graded isomorphism, $\Delta$ is equipped with an MT grading.

		Let $z''\in J\cap T_1$ be a homogeneous element with $(z'')_{(1,2)}=1$, and let $z'=z''\circ u_1$. The only non-zero entries of $z'$ can be $(1,2)$, $(n-1,n)$, $(1,n-1)$, $(2,n)$, and $z'$ is homogeneous with $(z')_{(1,2)}=1$. If $z=\frac12(z'\circ u_2+z'\circ e_2)$ then $z$ is homogeneous and $z=e_{12}+ae_{n-1,n}$ for some $a$. Since $\deg e_2\ne1$ we have $\deg z\ne\deg(z\circ e_2)$, hence  $a\ne0$. Let $A=\text{diag}(1,1,\ldots,1,a)$. Then $\psi\colon \text{UJ}_n\to \text{UJ}_n$ defined by $x\mapsto A x A^{-1}$ is an isomorphism. The induced grading is such that $\psi(z)=e_{12}+e_{n-1,n}$, $\psi(z\circ e_2)=e_{12}-e_{n-1,n}$ and $\psi(Y_{i:1}^\pm)=Y_{i:1}^\pm$, for $i=2$, 3, \dots, $\lceil\frac{n-1}2\rceil$. As the latter are homogeneous elements the induced grading is MT.

	Now, as in the previous Lemma one proves that whenever $\deg e_2\ne 1$ and $n=2$ or $3$, the grading is MT. When $n=3$ and $\deg e_2\ne1$, we can find a homogeneous element of type $z=e_{12}+ae_{23}$, where $a\in F$ is non-zero. Thus we can conjugate with the diagonal matrix $\text{diag}(1,1,a)$ as in the general case in order to obtain a MT-grading.
	\end{proof}

	\begin{Thm}
		Every $G$-grading on $\text{UJ}_n$ has commutative support and, up to a graded isomorphism, the grading is either elementary or MT.
	\end{Thm}

\section{On the graded identities}

	We have seen that non-isomorphic elementary gradings satisfy different graded identities. Let $G$ be a group and assume $A_1=(\text{UJ}_n,t_1,\eta)$ is an MT grading, and either $A_2=(\text{UJ}_n,t_2,\eta')$ with $t_1\ne t_2$ or $A_2=(\text{UJ}_n,\eta'')$ is an elementary grading. Then $f=(x_1^{(t_1)})^{\circ n}=x_1^{(t_1)}\circ x_1^{(t_1)}\circ\cdots\circ x_1^{(t_1)}$ is not a graded identity for $A_1$, but it is one for $A_2$.

	Now let $A_1=(\text{UJ}_n,t,\eta)$ and $A_2=(\text{UJ}_n,t,\eta')$, and assume $A_1\not\simeq A_2$. We shall use the notation of the proof of Lemma \ref{notation_bla}, and we will give an alternative proof for it.  If $\varphi\colon G\to G_0=G/\langle t\rangle$ is the canonical projection we denote the elementary gradings induced on $A_1$ and on $A_2$ by $\varphi$ as $\bar A_1=(\text{UJ}_n,\eta_0)$ and $\bar A_2=(\text{UJ}_n,\eta_0')$. Then we have two possibilities:

	\noindent(a) $\eta_0\ne\eta_0'$, hence we can find a polynomial $f(x_1^{(\bar g_1)},\ldots,x_m^{(\bar g_m)})$ such that $f$ is a graded identity for $\bar A_1$, but not for $\bar A_2$ (interchanging $\bar A_1$ and $\bar A_2$ if necessary). This means  that
\[
		g(x_1^{(g_1)}, x_1^{(g_1t)}, \ldots,x_m^{(g_m)}, x_m^{(g_mt)}) = f(x_1^{(g_1)}+x_1^{(g_1t)}, \ldots, x_m^{(g_m)}+x_m^{(g_mt)})
\]
	is a graded identity for $A_1$ but not for $A_2$.

	\noindent(b) $\eta_0=\eta_0'$. In this case, necessarily $n$ is even and, up to a graded isomorphism, $\eta=(g_1,\ldots,g_q)$ and $\eta'=(g_1,\ldots,g_{q-1},g_q')$ with $g_q\ne g_q'$. Note that $\text{Supp}\,J_1/J_1^2\ne\text{Supp}\,J_2/J_2^2$, where $J_i=(A_i,A_i,A_i)$, $i=1$ and $i=2$. Let $f=z_1^{(1)}\circ z_2^{(2)}\circ\cdots\circ z_q^{(q)}\circ z_{q-1}^{(q+1)}\circ\cdots\circ z_1^{(n-1)}$, where $z_i^{(j)}=(x_{3j-2}^{(1)},x_{3j-1}^{(1)},x_{3j}^{(g_i)})$. Then $f$ is a graded identity for $A_2$, but not for $A_1$.

	In this way we have the final result.
	\begin{Thm}
		Let $A_1$ and $A_2$ be two $G$-gradings on $\text{UJ}_n$. Then $A_1\simeq A_2$ as graded algebras if and only if $T_G(A_1)=T_G(A_2)$.
	\end{Thm}
	
\section*{Acknowledgements}
We are thankful to the Referee whose numerous remarks and comments helped us improve the exposition, and correct some arguments in the main proofs of the paper. The Example before Lemma~\ref{afterexample} was also a Referee's suggestion which we accepted with gratitude.

\end{document}